\documentclass[12pt,leqno,a4paper]{amsart}
\usepackage{amssymb,enumerate}
\usepackage{amsmath}

\usepackage{enumerate}
\usepackage{xcolor}

\newtheorem{theorem}{Theorem}[section]
\newtheorem{lemma}[theorem]{Lemma}
\newtheorem{corollary}[theorem]{Corollary}

\newtheorem*{thmA}{Theorem A}
\newtheorem*{thmB}{Theorem B}

\newtheorem*{thmC}{Theorem C}

\numberwithin{equation}{section}

\begin{document}
\title[Cosets of normal subgroups and union of two conjugacy classes]{Cosets of normal subgroups  and union of two conjugacy classes}

    \author[Beltr\'an]{
       Antonio Beltr\'an\\
     Departamento de Matem\'aticas\\
      Universitat Jaume I \\
     12071 Castell\'on, Spain\\
        }

 \thanks{ Antonio Beltr\'an: abeltran@uji.es,                                                                                   ORCID: 0000-0001-6570-201X }



\begin{abstract}
Let $G$ be a finite group, $N$ a normal subgroup of $G$ and $x\in G-N$. We discuss when the coset $Nx$ is contained in the union of two conjugacy classes, $K$ and $D$,  of $G$. We show that $N$ need not be solvable, and can even be non-abelian simple, but in these cases, $K$ and $D$ must have the same cardinality, and the non-solvable structure of $N$ is restricted. The non-abelian principal factors of $G$  contained in $N$ are  then isomorphic to $S\times \cdots\times S$, where $S$ is a simple group of Lie type of odd characteristic.
\end{abstract}

\maketitle

\section{Introduction}

Let $G$ be a finite group and $N$ a  proper normal subgroup of $G$. The set $G - N$ is the union of certain conjugacy classes of $G$, and at the same time, is the union of all non-trivial cosets of $N$. Guralnick and Navarro focused on a single coset of $N$ that is contained in a conjugacy class of $G$  and demonstrate the solvability of $N$ by appealing to the Classification of Finite Simple Groups \cite[Theorem B(a)]{GN}. Inspired by this result, the author proved that if a single coset $Nx$ with $x\in G - N$ lies in the union $K\cup K^{-1}$, where $K$ is a conjugacy class of $G$, then  $N$ is solvable as well \cite{B}.  Continuing this investigation, in this note we ask whether the same conclusion can be reached when $Nx$ is contained in the union of any two different conjugacy classes $K$ and $D$ (but not in just one of them). There are special cases where this is true, e.g., when  the elements of $K$ and $D$ are of  odd order \cite[Theorem B(c)]{GN}. In general, however,  we will see that the answer is negative.

\medskip
Our first result consists in translate the condition on the coset $Nx$, that  is,  $Nx\subseteq K \cup D$, into the language of  Character Theory.

\begin{thmA} Let $G$ be a finite group, $N$ a normal subgroup of $G$ and $x\not\in N$. Let $K =x^G$ and $D=d^G$ be distinct conjugacy classes of $G$. Then the following conditions  are equivalent:

 \begin{itemize}
 \item[(1)] $Nx \subseteq K \cup D$, but $Nx \not\subseteq K$.
 \item[(2)] Every  $\chi \in {\rm Irr}(G)$ satisfies that if $N\subseteq {\rm ker}(\chi)$, then $\chi(x)=\chi(d)$; and if $N\not\subseteq {\rm ker}(\chi)$, then
 $$|K|\chi(x) +|D|\chi(d)=0,$$
 and, in addition, $|{\bf C}_{G/N}(Nx)|= |G|/(|K|+|D|)$.
\end{itemize}
\end{thmA}

We remark that the conditions concerning character values, as well as the equality relating to the order of the centralizer in statement (2)  of Theorem A can be read off from the character table of $G$. Furthermore, under certain additional assumptions to the hypothesis of Theorem A,  more information about the cardinality of $K$ and $D$ and about the character values on these classes can be obtained.

\begin{thmB} Let $G$ be a finite group and $N$ a normal subgroup of $G$. Suppose that $K =x^G$ and $D=d^G$ are distinct conjugacy classes of $G$, with $x,d \in G$, such that $Nx\subseteq K \cup D$ and $Nx \not\subseteq K$.  Suppose further that $1_N\neq \theta \in {\rm Irr}(N)$ extends to $G$, say $\hat{\theta}_N=\theta$, with $\hat{\theta}\in {\rm Irr}(G)$. Then
\begin{itemize}
\item[(1)] $|K|=|D|$.
\item[(2)] Every extension $\chi\in {\rm Irr}(G)$ of $\theta$  satisfies  $\chi(x)=-\chi(d)$ and $|\chi(x)|=1$. In addition, every $\varphi\in {\rm Irr}(G)$ lying over $\theta$ is of the form $\varphi=\hat{\theta}\beta$
  with $\beta\in {\rm Irr}(G/N)$.
\item[(3)]
Every $\psi\in{\rm Irr}(G)$ that does not lie over $\theta$ or $1_N$ satisfies $\psi(d)=\psi(x)=0$. As a consequence, $\theta$ is the only non-principal irreducible character of $N$ that extends to $G$.
\end{itemize}
\end{thmB}

We show  through a series of examples that, under the assumptions of Theorem A, the subgroup $N$ may be solvable or not, and that $K$ and $D$ may possess different cardinality.
 However, when $N$ is non-solvable then the existence of a non-principal character of $N$  that extends to $G$ is guaranteed, and this, in turn, leads to the conclusion that $|K|=|D|$ and to specific constraints on the structure of $N$. Of course, our next result is based on the Classification.

\begin{thmC} Let $G$ be a finite group, $N$  a normal subgroup of $G$, and $x\not \in N$. Suppose that $K$ and $D$ are conjugacy classes of $G$ such that  $Nx\subseteq K \cup D$. If $N$ is non-solvable, then $|K|=|D|$ and  every non-abelian principal factor of $G$  contained in $N$ is isomorphic to $S\times \cdots \times S$, where $S$ is a simple group of Lie type of odd characteristic.
 \end{thmC}

All groups considered are finite, the employed notation and terminology are usual. We will use \cite{Huppert} and \cite{Isaacs} to refer to standard results on Character Theory.

\section{Proofs of Theorems A and B}

 We start by establishing some notation. Let $G$ be a finite group and $N$ a normal subgroup of $G$. We identify ${\rm Irr}(G/N)$ with the set of irreducible characters of $G$ that contain $N$ in its kernel, and denote by ${\rm Irr}(G|N)$ the set of irreducible characters of $G$ that do not contain $N$ in its kernel, so both subsets certainly provide a partition of ${\rm Irr}(G)$.

 \medskip
 In the complex group algebra ${\mathbb C}G$, if $X\subseteq G$, we write $\widehat{X} =\sum_{x\in X} x\in {\mathbb C}[G]$. Let $\{C_1,\ldots, C_t\}$ be the set of conjugacy classes of $G$. We know that each $\widehat{C_i}$ lies in ${\bf Z}( {\mathbb C}[G])$ and  that  $\{\hat{C_1},\ldots, \hat{C_t}\}$  is a basis of that algebra. Thus,  for every pair of conjugacy class sums
$\widehat{C}_{m}$ and $\widehat{C}_{n}$ with representatives $c_{m}$ and $c_{n}$ we can write
$$
\widehat{C}_{m}\widehat{C}_{n}=\sum_{i}\alpha_{i}\widehat{C}_{i},
$$
where
$$
 \alpha_{i}=\frac{|C_{m}||C_{n}|}{|G|}\sum_{\chi\in {\rm Irr}(G)}\frac{\chi(c_{m})\chi(c_{n})\overline{\chi(c_{i})}}{\chi(1)}
$$
 with $c_{i}\in C_{i}$ (see for  instance \cite[ Theorem 4.6]{Huppert}).
 These formulae will be used throughout without reference.

 \medskip
  For convenience, we prove  a reformulated version of Theorem A, which includes a little more information.

 \begin{theorem}\label{2} Let $G$ be a finite group and $N$ a normal subgroup of $G$. Let $K =x^G$ and $D=d^G$ be distinct conjugacy classes of $x,d \in G$. Then the following conditions are equivalent:

 \begin{itemize}
 \item[(a)] $Nx \subseteq K \cup D$, but $Nx \not\subseteq K$.
 \item[(b)] $NK=ND=K\cup D$.
 \item[(c)] Every  $\chi \in {\rm Irr}(G/N)$ satisfies $\chi(x)=\chi(d)$; and  every $\chi\in {\rm Irr}(G|N)$ satisfies
 $$|K|\chi(x) +|D|\chi(d)=0,$$
and, in addition, $|{\bf C}_{G/N}(Nx)|= |G|/(|K|+|D|)$.
 \end{itemize}

\end{theorem}

\begin{proof}  The  equivalence between (a) and (b) is elementary. Assume (a), so there exist $n\in N$ and $d^g\in D$ for some $g\in G$, such that $nx=d^g$. Hence $Nx=Nd^g=(Nd)^g$. It clearly follows that $K\cup D\subseteq NK=ND\subseteq K\cup D$, so the equality holds. Conversely, if we assume (b), we trivially have $Nx\subseteq K\cup D$. Now the fact that  $Nx\subseteq K$ yields to $NK=K$, against  our assumption. Therefore, $Nx\not\subseteq K$, as required.

\medskip
We prove now that (a) and (b) imply (c). As we have said above, the hypotheses imply that $Nx=(Nd)^g$ for some $g\in G$, so
it is evident that $\chi(x)=\chi(d)$ for every $\chi\in {\rm Irr}(G/N)$.
Suppose now that $\chi\in{\rm Irr}(G|N)$ and let $\mathfrak{X}$ be a representation of $G$ that affords $\chi$.   Also,  we know that $\mathfrak{X}$ can be linearly extended to $\mathbb{C}[G]$.   Since $N$ is a disjoint union of conjugacy classes of $G$, then $$\widehat{ N} = \sum _{n \in N} n\in {\bf Z}(\mathbb{C}[G])$$ so, by Schur's Lemma, $\mathfrak{X}(\widehat{N})$ is a scalar matrix. On the other hand, the trace of $\mathfrak{X}(\widehat{N})$ is $$\sum_{n \in N}\chi(n)=|N|[\chi_N, 1_N] =0,$$ so $\mathfrak{X}(\widehat{N}) = O$, where $O$ denotes the zero matrix. Then $\mathfrak{X}(\widehat{Ng})=\mathfrak{X}(\widehat{N})\mathfrak{X}(g)=O$, for every $g\in G$. Now, by hypothesis, we can write $\widehat{K}\widehat{N} = m_1\widehat{K}+m_2\widehat{D}$ with $m_1, m_2>0$, and we know that $\mathfrak{X}(\widehat{N}\widehat{K})=\sum_{x^g \in K} \mathfrak{X}(\widehat{Nx^g})=O$.
Hence, by taking traces, we have
\begin{equation}
0={\rm Trace}(\mathfrak{X}(\widehat{N}\widehat{K}))=m_1 {\rm Trace}(\widehat{K}) + m_2 {\rm Trace}(\widehat{D})=m_1|K| \chi(x) + m_2|D| \chi(d). \label{eqn1}
\end{equation}

We will prove that $m_1=m_2$, and thus,  we will obtain $|K| \chi(x) +|D| \chi(d)=0$.
The second orthogonality relation together with the fact that  $\chi(x)=\chi(d)$ for every $\chi\in {\rm Irr}(G/N)$  give rise to the equations:

\begin{equation}
|{\bf C}_G(x)|=\sum_{\chi\in {\rm Irr}(G/N)} |\chi(x)|^2
+ \sum_{\chi\in {\rm Irr}(G|N)}
 \chi(x)\overline{\chi(x)}=|{\bf C}_{G/N}(Nx)|+ \sum_{\chi\in {\rm Irr}(G|N)}|\chi(x)|^2,  \label{eqn2}
 \end{equation}

\begin{equation} 0=\sum_{\chi\in {\rm Irr}(G/N)} |\chi(x)|^2
+ \sum_{\chi\in {\rm Irr}(G|N)}
 \chi(x)\overline{\chi(d)} =
|{\bf C}_{G/N}(Nx)| + \sum_{\chi\in {\rm Irr}(G|N)} \chi(x)\overline{\chi(d)}. \label{eqn3}
 \end{equation}
By applying Eq. (\ref{eqn1}) to Eq. (\ref{eqn3}), we get

\begin{equation}\label{eqn4}
0=|{\bf C}_{G/N}(Nx)| - \frac{m_1|K|}{m_2|D|}\sum_{\chi\in {\rm Irr}(G|N)}
 |\chi(x)|^2.
 \end{equation}

\noindent
On the other hand, the  equality $NK=K \cup D$ allows to write
\begin{equation}
Nx_1\cup \cdots \cup Nx_t= K\cup D \label{eqn5}
\end{equation}
where $Nx_i$ for $i=1,\ldots t$ are the distinct conjugates of $Nx$  in $G/N$, that is, $t=|G/N:{\bf C}_{G/N}(Nx)|$. Taking cardinalities in Eq. (\ref{eqn5}), we get $|N|t= |K|+ |D|$, and hence
$$|{\bf C}_{G/N}(Nx)|=  \frac{|G|}{|K|+ |D|}.$$
Thus, the last equality of (c) is proved.
We now substitute this value, together with the equality $|{\bf C}_G(x)|=|G|/|K|$, into  Eqs. (\ref{eqn2}) and (\ref{eqn4}),  and this gives
$$0= \frac{|G|}{|K|+ |D|}-\frac{m_1|K|}{m_2|D|}\left(\frac{|G|}{|K|}-\frac{|G|}{|K|+ |D|}\right).$$

\noindent
It follows that $m_1=m_2$, as required.

\medskip
In order to prove that (c) implies (b), we show first that $KC\subseteq K\cup D$ for every conjugacy class $C$ of $G$ contained in $N$. Write
 \begin{equation}\label{eqn6}
 \widehat{K}\widehat{C}=\alpha_K\widehat{K}+\alpha_D\widehat{D} + \sum_{i=1}^n \alpha_i \widehat{C_i},
 \end{equation}
 $C_1, \ldots , C_n$ being  all the conjugacy classes of $G$ (including the trivial class) distinct from $K$ and $D$, and $\alpha_K,\alpha_D, \alpha_i$ are non-negative integers. We compute  $\alpha_K|K| + \alpha_D|D|$.

 $$ \alpha_K|K| + \alpha_D|D|=\frac{|K|^2|C|}{|G|}\sum_{\chi\in {\rm Irr}(G)}\frac{\chi(x)\chi(c)\chi(x^{-1})}{\chi(1)}+
 \frac{|D||K||C|}{|G|}\sum_{\chi\in {\rm Irr}(G)}\frac{\chi(x)\chi(c)\chi(d^{-1})}{\chi(1)}=$$
 $$=\frac{|K|^2|C|}{|G|}\sum_{\chi\in {\rm Irr}(G)}\frac{|\chi(x)|^2\chi(c)}{\chi(1)}+
 \frac{|D||K||C|}{|G|}\sum_{\chi\in {\rm Irr}(G)}\frac{\chi(x)\chi(c)\overline{\chi(d})}{\chi(1)}= $$
 $$=\frac{|K|^2|C|}{|G|}(\sum_{\chi\in {\rm Irr}(G/N)}|\chi(x)|^2 + \sum_{\chi\in {\rm Irr}(G|N)}\frac{|\chi(x)|^2\chi(c)}{\chi(1)}) +$$
 $$+  \frac{|D||K||C|}{|G|}(\sum_{\chi\in {\rm Irr}(G/N)} \chi(x)\overline{\chi(d)} + \sum_{\chi\in {\rm Irr}(G|N)}\frac{\chi(x)\chi(c)\overline{\chi(d})}{\chi(1)}).$$

\noindent
By employing the  three equations appearing in statement (c), we have

  $$ \alpha_K|K| + \alpha_D|D|= (\frac{|K|^2|C|}{|G|} +
  \frac{|D||K||C|}{|G|})|{\bf C}_{G/N}(Nx)| +$$
 $$+ \frac{|K|^2|C|}{|G|}\sum_{\chi\in {\rm Irr}(G|N)}\frac{|\chi(x)|^2\chi(c)}{\chi(1)} +\frac{|D||K||C|}{|G|}
  \sum_{\chi\in {\rm Irr}(G|N)}(-\frac{|K|}{|D|})\frac{|\chi(x)|^2\chi(c)}{\chi(1)}=$$
$$=( \frac{|K|^2|C|}{|G|} +
  \frac{|D||K||C|}{|G|})\frac{|G|}{|K|+|D|}=|K||C|.$$
  
 On the other hand,  by counting elements in Eq. (\ref{eqn6}), we obtain
 $$|K||C|=\alpha_K|K|+ \alpha_D|D| + \sum_{i=1}^n\alpha_i|C_i|.$$
By joining the above results, we deduce that $\alpha_i=0$ for  $i=1,\ldots , n$. This means that $KC\subseteq K\cup D$, as required. We remark that $KC=K$ is feasible when $\alpha_D=0$ and $\alpha_K=|C|$ (for instance, when $C=1$),  and $KC=D$ can also occur (when $\alpha_K=0$ and $\alpha_D=|K||C|/|D|)$.
Accordingly, we conclude that $KN\subseteq K \cup D$. Finally, the fact that $\chi(x)=\chi(d)$ for every $\chi\in {\rm Irr}(G/N)$  is equivalent (because ${\rm Irr}(G)$ is a basis of the vector space of class functions of $G$) to assert that $Nx$ and $Nd$ are conjugate, and consequently, $NK=ND$. This obviously implies that $K\cup D\subseteq KN$, so the equality of (b) holds.
\end{proof}

We remark that  if the assumption  that all elements in $Nx$ are not $G$-conjugate  is eliminated in  statement (1) of  Theorem A, then the result varies somewhat. In fact, when $Nx\subseteq K$, where $K$ is the conjugacy class of $x$, it turns out that $\chi(x)=0$ for every $\chi \in {\rm Irr}(G)$ with $N$  not contained in its kernel (see for instance  \cite[Lemma 3.1]{GN}). We want to show that, indeed, this  condition is equivalent.

\begin{theorem} Let $G$ be a finite group and $N$ a normal subgroup of $G$. Let $x\not\in N$ . Then all the
elements in $Nx$ are $G$-conjugate if and only if $\chi(x)=0$ for every $\chi \in {\rm Irr}(G)$ that does not contain $N$ in its kernel.
\end{theorem}

\begin{proof} As we said above, the direct sense is known. Conversely, assume that $\chi(x)=0$ for every $\chi \in {\rm Irr}(G|N)$. Note first that the second ortogonality relation  provides the equality $|{\bf C}_G(x)|=|{\bf C}_{G/N}(Nx)|$.  Let $C$ be an arbitary conjugacy class of $G$ contained in $N$ and write
\begin{equation} \label{eqn22}
\widehat{K}\widehat{C}=\alpha_K \widehat{K} + \sum_{i=1}^n \alpha_i\widehat{C_i},
\end{equation}
where $C_i$  for $i=1,\ldots , n$ are all the conjugacy classes of $G$ distinct from $K$, and $\alpha_K$ and $\alpha_i$ are non-negative integers. Taking into account the hypothesis, we have
$$\alpha_K = \frac{|K||C|}{|G|}\sum_{\chi\in {\rm Irr}(G)}\frac{\chi(x)\chi(c)\overline{\chi(x)}}{\chi(1)}= \frac{|K||C|}{|G|} \sum_{\chi\in {\rm Irr}(G/N)}|\chi(x)|^2= $$
$$=\frac{|K||C|}{|G|} |{\bf C}_{G/N}(Nx)|= \frac{|K||C|}{|G|}|{\bf C}_G(x)|=|C|.$$

On the other hand, by counting elements in Eq. (\ref{eqn22}), we have
$$|K||C|=\alpha_K|K| +\sum _{i=1}^n \alpha_i|C_i|.$$
 This implies that $\alpha_i=0$ for every $i=1,\ldots, n$, or equivalently, $KC=K$. As this equality holds for every class $C$ of $G$ contained in $N$, we deduce that $KN=K$. This is clearly equivalent to state that all elements in $Nx$ are $G$-conjugate.
 \end{proof}

We demonstrate now Theorem B.

\begin{proof}[Proof of Theorem B]
 Let $\chi\in{\rm Irr}(G)$ be an extension of $\theta\in {\rm Irr}(N)$, with $\theta\neq 1_N$, that is, assume $\chi_N=\theta$. By   \cite[Lemma 8.14(c)]{Isaacs}, we know that
\begin{equation}
|N|=\sum_{y\in Nx} \chi(y)\overline{\chi(y)}. \label{eqn7}
\end{equation}
Now, by Theorem \ref{2}, we have $KN=K\cup D$, and considering  $KN$ as the union of cosets of $N$, we can write 
\begin{equation} \label{eqn8}
Nx_1\cup \cdots \cup Nx_t= K\cup D
\end{equation}
where $Nx_i$ for $i=1,\ldots t$ are the distinct conjugates of $Nx$  in $G/N$, and $t=|G/N:{\bf C}_{G/N}(Nx)|$.
By taking cardinalities, it follows that $|N|t=|K|+|D|$.  By applying Eq. (\ref{eqn7}) to $Nx_i$ for every  $i=1,\ldots t$, we obtain
$$|N|t=\sum_{i=1}^t \sum_{y\in Nx_i} \chi(y)\overline{\chi(y)}= \sum_{y\in K \cup D} \chi(y)\overline{\chi(y)}=|K||\chi(x)|^2 + |D||\chi(d)|^2.$$
Now, obviously  $\chi\in {\rm Irr}(G|N)$, so by  Theorem \ref{2}(c), we have $\chi(d)=-(|K|/|D|)\chi(x)$. Since  $|N|t=|K|+|D|$, it follows that
$$|K|+|D|= (|K|+ \frac{|K|^2}{|D|})|\chi(x)|^2.$$
This implies that $|\chi(x)|^2= |D|/|K|$, and  consequently, $|\chi(d)|^2=|K|/|D|$. Now,  both values are algebraic integers for $\chi(x)$ and $\chi(d)$ being algebraic integers too,  and moreover, they are rational, so they are necessarily integers. This forces $|K|=|D|$, and hence,  $|\chi(x)|=1$ and $\chi(x)=-\chi(d)$.  Also, if  $\varphi\in {\rm Irr}(G)$ lies over $\theta$, then by  Gallagher's theorem \cite[Corollary 6.17]{Isaacs}, we know that $\varphi=\beta \hat{\theta}$ with $\beta \in {\rm Irr}(G/N)$. Thus, (1) and (2) are proved.

\medskip
To prove (3), suppose that $\psi\in{\rm Irr}(G)$  satisfies $[\psi_N, 1_N]= [\psi_N, \theta]=0$. Let us consider again $\chi\in{\rm Irr}(G)$  any extension of $\theta$, and view $1_G$ as the trivial extension of $1_N$. We apply  twice \cite[Lemma 8.14(b)]{Isaacs}, so
$$0=\sum_{y\in Ng} \psi (y)\overline{\chi(y)}, \quad   0=\sum_{y\in Ng} \psi (y)\overline{1_G(y)}$$
for every coset $Ng$. Now,  taking into account the decomposition (\ref{eqn8}) and using the above equations, it follows that

$$0=\sum_{i=1}^t \sum_{y\in Nx_i} \psi(y)\overline{\chi(y)}= \sum_{y\in K \cup D} \psi(y)\overline{\chi(y)}=|K|\psi(x)\overline{\chi(x)} + |D|\psi(d)\overline{\chi(d)},$$

$$0=\sum_{i=1}^t \sum_{y\in Nx_i} \psi(y)\overline{1_G(y)}= \sum_{y\in K \cup D} \psi(y)=|K|\psi(x) + |D|\psi(d).$$
Since we know that $|K|=|D|$ and  $\chi(x)=-\chi(d)\neq 0$ by (1) and (2),  the above equations force $\psi(x)=\psi(d)=0$, so  the first assertion of (3) is proved. Finally, if some $\beta\in {\rm Irr}(N)$, with $\beta \neq 1_N, \theta$, extends to $\psi\in{\rm Irr}(G)$, we certainly have  $[\psi_N, 1_N]=[\psi_N, \theta]=0$ (by Clifford's theorem \cite[Theorem 6.2]{Isaacs}), and then, we have proved above that $\psi(x)=\psi(d)=0$. But by the first part of the proof, we know that $|\psi(x)|=1$, providing a contradiction. This shows that $\theta$ is the only  non-principal irreducible character of $N$ that extends to  $G$.
\end{proof}

\section{Proof of Theorem C}

We present some previously developed results that are necessary for our purposes, the first of which has already been quoted in the Introduction.

\begin{theorem}{\normalfont  \cite[Theorem B(a)]{GN}}\label{11}
Suppose that $N$ is a normal subgroup of a finite group $G$. Let $x\in G$. If all the elements of $Nx$ are $G$-conjugate, then $N$ is solvable.
\end{theorem}

\begin{lemma} {\normalfont \cite[Lemma 2]{BCFM}}\label{12}
 Let $G$ be a finite group and $K, C$ and $D$ non-trivial conjugacy classes of $G$ such that
 $KC=D$ with $|K|=|D|$. Then $\langle C\rangle$ is solvable.
 \end{lemma}

 The proof of Theorem C is reduced to the case when $N$ is the direct product of non-abelian simple groups. In that case, the existence of irreducible characters of coprime degrees in certain non-abelian simple groups  that extend to the automorphism group of $N$ is required.  With regard to the alternating and sporadic  simple groups  the following is established.

\begin{theorem} {\normalfont  \cite[Theorem 3]{BCLP}}\label{3}
If $n\geq 6$, then  {\rm Alt($n$)} has irreducible characters of degree $(n-1)(n-2)/2$ and  $n(n-3)/2$ that extend to {\rm Sym($n$)}.
 \end{theorem}

 Observe that $(n-1)(n-2)/2$ and  $n(n-3)/2$ are coprime
numbers.

\begin{theorem} {\normalfont  \cite[Theorem 4]{BCLP}} \label{4} Let $S$ be a sporadic simple group or the Tits group and let $A$ be the automorphism group of $S$. Then there exist non-linear characters $\chi_m,\chi_n \in{\rm Irr}(S)$ such that $(\chi_m(1),\chi_n(1))=1$ and that both $\chi_m$ and $\chi_n$ extend to $A$.
\end{theorem}

 According to the Atlas \cite{atlas},  Table 1 of \cite{BCLP}   enumerates the specific characters of Theorem \ref{4}, along with their respective degrees. Nevertheless, these degrees will not be utilized in our analysis.

\medskip
 In the context of simple groups of Lie type, in general, the existence of two non-linear irreducible characters of coprime degree that extend to the entire automorphism group is not generally valid.  Example 4 provides instances of it. Indeed, we will get more information in Theorem \ref{32}(c) under our assumptions. Instead, we will use some well-known properties of the Steinberg character. The reader is referred to \cite{Carter} for its standard properties. The Steinberg character of a simple group of Lie type $S$ is a (rational) character that extends to  Aut($S$).  This was first proved by Schmid in \cite{Schmid1, Schmid2}. Furthermore,  the Steinberg character extends to the character of a rational representation of Aut($S$) (see \cite{Feit}).
   By using tensor inductions of characters, the authors of \cite{BCLP} also demonstrated the following extensibility property, needed in our developement, related to normal subgroups that are direct product of simple groups.

\begin{theorem} {\normalfont  \cite[Lemma 5]{BCLP}} \label{5}
Let $N$ be a minimal normal subgroup of a group $G$ so that $N = S_1 \times \ldots \times S_t$, where $S_i\cong S$ is a  non-abelian simple group. Let $A$ be the automorphism group of $S$. If $\sigma \in {\rm Irr}(S)$ extends to $A$, then $\sigma \times \ldots \times \sigma \in {\rm Irr}(N)$ extends to $G$.
\end{theorem}

Next we give a characterization, in terms of irreducible characters, of an equality involving the product of two of conjugacy classes, which generalizes \cite[Lemma 3.1]{B}. This kind of equation arises when considering the hypotheses of Theorem C, and thereby, is useful in the proof.

\begin{lemma} \label{31} {\it Let $G$ be a finite group and $x, d, c\in G$. Let $K=x^G, D=d^G$ and $C
=c^G$ be the corresponding conjugacy classes of $G$, with $K\neq D$. Then, the following conditions are equivalent:}
\begin{itemize}
\item[(1)] $KC= K \cup D$.
\item[(2)] For every $\chi \in$ {\rm Irr}$(G)$ $$|K||C|\chi(x)\chi(c)=\chi(1)(a|K| \chi(x)+b|D|\chi(d)),$$
for some integers $a,b> 0$ such that $a|K|+b|D|=|K||C|$.
In particular, when $|K|=|D|$ then $a+b=|C|$, and
$$|C|\chi(x)\chi(c)=\chi(1)(a \chi(x)+b\chi(d)).$$
\end{itemize}
\end{lemma}

{\it Proof}. Suppose  first that $KC= K\cup D$.  Let  us denote by $\widehat{K},\widehat{D}$ and $\widehat{C}$  the sum classes in the center of  ${\mathbb C}[G]$, so we can  write $\widehat{K}\widehat{C}=a\widehat{K}+b\widehat{D}$ for some integers $a,b> 0$. Notice that, by counting elements, $|K||C|= a|K|+ b|D|$. Also, by  \cite[Lemma 3.8 and Theorem 3.9]{Huppert}, we know that
$$\frac{|K|\chi(x)}{\chi(1)} \frac{|C|\chi(c)}{\chi(1)}= a \frac{|K|\chi(x)}{\chi(1)}+b\frac{|D|\chi(d)}{\chi(1)}$$
 for every $\chi \in$ Irr$(G)$, and hence,  the first equality  stated in (2)  is proved. Of course, the consequence when $|K|=|D|$ is straightforward.

\medskip
Conversely, assume that (2) holds. If the $C_i's$ denote all the conjugacy classes of $G$ and $c_i\in C_i$, we write

\begin{equation}\label{1}
\widehat{K}\widehat{C}=\sum_{i}\alpha_{i}\widehat{C}_{i} \, \, \, \, \, {\rm with} \, \, \, \, \, \alpha_{i}=\frac{|K||C|}{|G|}\sum_{\chi\in {\rm Irr}(G)}\frac{\chi(x)\chi(c)\chi(c_{i}^{-1})}{\chi(1)}.
 \end{equation}
 On the other hand, by isolating $\chi(x)\chi(c)$ from the equation given in (2), we get
 $$
 \chi(x)\chi(c)=\frac{\chi(1)(a|K|\chi(x)+b|D|\chi(d))}{|K||C|},
 $$
  and by replacing it in Eq. (\ref{1}), we have
  $$
 \alpha_{k}=\frac{|K||C|}{|G|}\sum_{\chi\in {\rm Irr}(G)} \frac{(a|K|\chi(x)+b|D|\chi(d))\chi(c_i^{-1})}{|K||C|}=$$
$$ =\frac{1}{|G|}(a |K| \sum_{\chi\in {\rm Irr}(G)} \chi(x)\chi(c_i^{-1})+b |D| \sum_{\chi\in {\rm Irr}(G)} \chi(d)\chi(c_i^{-1})).
$$

By using the second orthogonality relation we deduce that $\alpha_{i}= a$ when $C_{i}=K$; that $\alpha_{i}= b$ when $C_{i} =D$; and that $\alpha_{i}=0$ otherwise. This implies that $KC=K \cup D$,  so (1) is proved. $\Box$

\medskip

The proof of Theorem C reduces to the case when $N$ is a direct product of isomorphic non-abelian simple groups, so we need first to deal with this case.

\begin{theorem}\label{32}
 Suppose that $G$ is a finite group, $N$  a non-abelian minimal normal subgroup of $G$ and $K=x^G$ and $D=d^G$ are conjugacy classes of $G$ such that $Nx\subseteq K\cup  D$. Then the following properties hold.
 \begin{itemize}
 \item[(1)] $|K|=|D|$.
 \item[(2)] $N\cong S\times \cdots \times S$, where $S$ is a simple group of Lie type of odd characteristic.
 \item[(3)] Let $\mathsf{St}$ be the Steinberg character of $S$. Then $\theta= \mathsf{St}\times \ldots \times \mathsf{St}\in {\rm Irr}(N)$  is the unique non-trivial irreducible character of $N$ that extends to $G$.
 \end{itemize}
\end{theorem}

\begin{proof}  We first observe that $Nx\not\subseteq K$, otherwise Theorem \ref{11} would imply the solvability of $N$, against our hypotheses. We view ${\rm Alt}(5)\cong {\rm PSL}(2, 4)$ and ${\rm Alt}(6)\cong {\rm PSL}(2, 9)$ as simple groups of Lie type. According to  Theorems \ref{3} and \ref{4}, and the aforementioned property that the Steinberg character of a simple group of Lipe type $S$ extends to Aut$(S)$, joint with  Theorem \ref{5},  it can be affirmed that there exists $\chi \in {\rm Irr}(G)$  that extends some $\phi\in {\rm Irr}(N)$ with $\phi\neq 1_N$.  Therefore, by Theorem B(1),  we have $|K|=|D|$.

\medskip
We know that $NK=K \cup D$ by Theorem \ref{2}. Then, for every non-trivial conjugacy class $C$ of $G$ contained in $N$, we have $KC\subseteq K \cup D$. If $KC=K$ or $KC=D$, then  the solvability of $\langle C\rangle= N$ follows  by Theorem \ref{11} and Lemma \ref{12}, respectively, contradicting our hypothesis. Therefore, $KC=K\cup D$ for every $C\neq 1$, so we are in a position to apply
 Lemma \ref{31}.
\medskip

We  analyze now each of the distinct possibilities for $S$ given by the Classification of Finite Simple Groups.
Suppose first that $S$ is isomorphic to an alternating  or a sporadic simple group. As we said, ${\rm Alt}(5)$ and ${\rm Alt}(6)$ will be treated as groups of Lie type. Then, by Theorems \ref{3} and \ref{4}, there exist two (nonlinear) irreducible characters of $S$, say $\theta_1$ and $\theta_2$, having  coprime degrees that extend to Aut$(S)$ (the automorphism group of ${\rm Alt}(n)$ is ${\rm Sym}(n)$ when $n\geq 7$). By Theorem \ref{5},  the characters $\theta_j\times \ldots \times \theta_j\in {\rm Irr}(N)$ for $j=1,2$ also extend to $G$. Let $\chi_1$, $\chi_2\in {\rm Irr}(G)$ denote two respective extensions of $\theta_j\times \ldots \times \theta_j$ for $j=1,2$.
 If we fix a non-trivial class $C$ of $G$ contained in $N$ and take $c\in C$,  by
 Lemma \ref{31},  there exist integers $a, b> 0$ with $a+b=|C|$ such that
$$|C|\chi_j(x)\chi_j(c)= \chi_j(1)(a \chi_j(x) + b \chi_j(d)),
$$
for $j=1,2$. Since we have proved in Theorem B(2) that $\chi_j(x)=-\chi_j(d)\neq 0$ for $j=1,2$, this yields
$$\chi_j(c)= \frac{(a-b)\chi_j(1)}{|C|} \quad {\rm for \quad } j=1,2.$$
Now, $\chi_1(1)$ and $\chi_2(1)$ are coprime integers, and hence, there exist $m,n\in {\mathbb Z}$ such that
$$m\chi_1(1)+n\chi_2(1)=1.$$
It follows that
$$m\chi_1(c)+ n\chi_2(c)=m \frac{(a-b)\chi_1(1)}{|C|} + n\frac{(a-b)\chi_2(1)}{|C|}= \frac{a-b}{|C|}$$
is an algebraic integer and rational, so it must be an integer. However,  as we are assuming $K\neq D$, then $0<a,b<|C|$, and this leads to a contradiction unless $a=b$. But if $a=b$  for every class $C\neq 1$ of $G$ contained in $N$, then $(\theta_j\times\ldots\times \theta_j)(n)= \chi_j(n)=0$ for $j=1,2$ and for every $1\neq n\in N$, and this is not possible. This contradiction  allows us to exclude the alternating groups (for $n\geq 7$) and the sporadic simple groups.

\medskip
Assume now that $S$ is a simple group of Lie type of characteristic $2$  (we include ${\rm Alt}(5)$ here) and let  $\hat{\theta}\in {\rm Irr}(G)$ be an extension of $\theta= \mathsf{St}\times \ldots \times \mathsf{St}\in {\rm Irr}(N)$. It is well-known that $\mathsf{St}$ vanishes on $2$-singular elements of $S$, and hence, $\hat{\theta}(c)=0$ for every $c\in N$ of even order. Say, for instance,  $c\in C$ for some class $C$ of $G$.  Arguing as above, by using Lemma \ref{31}, there exist integers $a,b>0$ with $a+b= |C|$ such that
$$0=\chi(1)(a\chi(x)+ b\chi(d)).$$
However, we have noted above that $\chi(x)=-\chi(d)$ is non-zero, so this gives $a=b$. Thus, $a=b=|C|/2$, and this means that every conjugacy class of $G$ of elements of even order of $N$ must have even cardinality. But this is simply not true. Indeed, if we take $P$ a Sylow 2-subgroup of $G$, then $N\cap P\unlhd P$, and  ${\bf Z}(P)\cap N\neq 1$. Then,  every $1\neq c\in {\bf Z}(P)\cap N$ satisfies that $|c^G|$ is odd.
 Therefore, the only possibility for $S$ is to be a simple group of Lie type with odd characteristic, and thus the proof of (2) is finished.

\medskip
Finally, we have already reasoned that   $\theta= \mathsf{St}\times \ldots \times \mathsf{St}\in {\rm Irr}(N)$ extends to $G$, so  assertion (3) follows straightforwardly by Theorem $B$.
\end{proof}

\noindent
{\it Remark}. The group ${\rm Alt}(6)$, viewed as ${\rm PSL}(2,9)$, is the only alternating group included in the thesis of Theorem \ref{32}, and likewise, in the thesis of Theorem C. In fact, Example 4 shows that this case may happen.
 
\medskip
Theorem B provides more information on the values that the irreducible characters of $G$ take on the conjugacy classes $K$ and $D$, appearing in the statement of Theorem \ref{32}.

\begin{corollary}\label{33}
Suppose that $G$ is a finite group, $N$  a non-abelian minimal normal subgroup of $G$ and $K=x^G$ and $D=d^G$ are conjugacy class of $G$ such that $Nx\subseteq K\cup  D$. Then $N=S\times\cdots\times S$, with $S$ a simple group of Lie type with odd characteristic and

  \begin{itemize}
  \item[(a)]  For every $\chi\in {\rm Irr}(G/N)$, we have $\chi(x)=\chi(d)$.
  \item[(b)] Let $\hat{\theta}\in {\rm Irr}(G)$ denote any extension of  $\theta= \mathsf{St}\times \ldots \times \mathsf{St}\in {\rm Irr}(N)$  to $G$.  Then $|\hat{\theta}(x)|= |\hat{\theta}(d)|=1$ and $\hat{\theta}(x)=- \hat{\theta}(d)$. In addition, for every $\chi\in {\rm Irr}(G)$ lying over $\theta$, we have $\chi=\hat{\theta}\beta$
 for some $\beta\in {\rm Irr}(G/N)$.
  \item[(c)]  For every $\chi\in {\rm Irr}(G)$ that does not lie over $\theta$ or $1_N$, we have $\chi(x)=\chi(d)=0$.
\end{itemize}
\end{corollary}

\begin{proof}
It is immediate from Theorem \ref{32} and Theorem B.
\end{proof}

We are ready to prove  Theorem C.

\medskip
\begin{proof}[Proof of Theorem C]  Suppose that $L/M$ is a non-abelian principal factor of $G$ such that $L\subseteq N$, and use bars to denote the factor group $G/M$. From the hypotheses, we certainly have $\overline{Lx}\subseteq \overline{Nx}\subseteq \overline{K}\cup \overline{D}$. Then, since $\overline{L}$ is a (non-abelian) minimal normal of $\overline{G}$, we can apply Theorem \ref{32} to conclude that $\overline{L}$ is a direct product of isomorphic simple groups of Lie type with odd characteristic. Thus, the second claim of the theorem is proved. Also, we know   that  $\theta\in {\rm Irr}(\overline{L})$, constructed from the Steinberg character of a direct factor of $\overline{L}$, extends to some $\chi\in {\rm Irr}(\overline{G})$. Now, viewing $\theta$ and $\chi$ as irreducible characters of $L$ and $G$, respectively, we certainly have that $\chi_N\in {\rm Irr}(N)$  extends to $G$. We can apply  then  Theorem B(1) to conclude that $|K|=|D|$, so the proof is finished.
 \end{proof}

\section{Examples}

We include a series of examples illustrating  different possibilities for the cardinalities of the conjugacy classes and  for the solvability or non-solvability of $N$.  For each example, we  provide the parameters in  the Small Groups library or  other libraries of  Classical Groups of GAP \cite{GAP}. This might allow the reader to verify them more easily. Example 4 can be also checked with the help of \cite{atlas}.

\medskip
\noindent
{\bf Example 1}.  Under the  conditions of Theorem A and assuming further that $N$ is solvable, the equality $|K|=|D|$ does not necessary hold.  Let $G=C_2\times {\rm Alt}(4)$ ({\it SmallGroup}(24,23)), $C_2\times C_2\cong N\leq {\rm Alt}(4)$ and $1\neq x\in{\bf Z}(G)$. Then $Nx=K\cup D$, where $K=\{x\}$ and $|D|=3$. No irreducible character of $N$  distinct from the principal one extends to $G$, so Theorem B does not hold either. Note that for every $\chi\in{\rm Irr}(G|N)$, we have
 $$\chi(x)+ 3\chi(d)=0,$$
just as Theorem A claims. An example of similar features  with trivial center is $G=C_2^3.A_4$, the second non-split extension by $C_2^3$ of ${\rm Alt}(4)$ acting faithfully ({\it SmallGroup}(96,72)),  $N$ being the only normal subgroup of order $4$ of $G$. By choosing $x$ of order $2$ lying in the only conjugacy class of size $4$, we have $Nx\subseteq K \cup D$, where $K=x^G$ and $D$ is the unique class of size $12$. 

\medskip
\noindent
 {\bf Example 2}. Here is an example of a (non-abelian) solvable normal subgroup satisfying the hypotheses of Theorem $A$ with $|K|=|D|$. Take $G={\rm SL}(2,3)$ ({\it SmallGroup}(24,3)), $N\cong Q_8$ and $x$ an element of order $3$. In fact, there exist exactly  two conjugacy classes of elements of order $3$ out of $N$, and both classes work. Let $1\neq z\in {\bf Z}(G)$. We have $Nx=x^G \cup (xz)^G$ with $|x^G|=|(xz)^G|= 4$. We  note that only $1_N$ and the irreducible character of degree $2$ of $N$ extend to $G$, so this example verifies the assumptions of Theorem B.

\medskip
\noindent
{\bf Example 3}. Another example of a solvable normal subgroup that meets the hypotheses of Theorem B is $G=A\Gamma L(8)\cong (C_2^3 \rtimes C_7)\rtimes C_3$ ({\it SmallGroup}(168,43)), $N=C_2^3\rtimes C_7$ and $x$ any element of order $3$ of $G$. Then $Nx=K\cup D$ where $K=x^G$ and $D$ is one of the two conjugacy classes of elements of order $6$, with $|K|=|D|=28$. The only non-trivial irreducible character of $N$ that extends to $G$ is the irreducible character of degree $7$.

\medskip
\noindent
{\bf Example 4}.
Let $G={\rm P\Gamma L}(2,9)$ ({\it PGammaL}(2,9)) be the projective general semilinear group,  $N= {\rm PSL}(2,9)$, and $x\in G- N$ an element of order $4$ such that $|x^G|=180$ (there exists another conjugacy class of elements of order 4 in $G-N$, but this does not work). This group, with $N$ non-abelian simple,  satisfies the conditions of Theorems A, B and C.  We have $Nx=K \cup D$, where $K=x^G$ has size $180$ and $D=d^G$ is the unique class of elements of order $8$ of size $180$.  The classes $K$ and $D$ correspond to ``4b" and ``8a", according to the GAP notation. As claimed in Corollary \ref{33}, the Steinberg character is the only non-trivial irreducible character of $N$ that extends to $G$. The group $G={\rm P\Gamma L}(2,27)$, $N= {\rm PSL}(2,27)$, with $x$ and $d$ appropiated elements of order $6$ and $12$  respectively (corresponding to either  the clases ``6a" and ``12a", or to ``6b" and ``12b", following the notation in \cite{GAP}), provide two analogue examples.


\bigskip
\noindent
{\bf Acknowledgements}.  This work is supported by the National Natural Science Foundation of China (No. 12071181).

\bibliographystyle{plain}

\end{document}